\documentclass[11pt]{article}
\usepackage{amsmath,amssymb,amsthm}
\usepackage{graphicx}
\usepackage{xcolor}
\usepackage{tikz}
\usetikzlibrary{shapes.geometric}
\DeclareUnicodeCharacter{3010}{}
\DeclareUnicodeCharacter{3011}{}
\DeclareMathOperator{\diam}{diam}

\usepackage{verbatim}

\usepackage[colorlinks=true, linkcolor=blue, citecolor=blue, urlcolor=blue]{hyperref}
\newtheorem{theorem}{Theorem}
\newtheorem{proposition}{Proposition}
\newtheorem{lemma}{Lemma}
\theoremstyle{definition}
\newtheorem{definition}{Definition}

\theoremstyle{remark}
\newtheorem{remark}{Remark}

\begin{document}

\begin{center}
{\Large \bf Quantization Errors, Human--AI Interaction, and Approximate Fixed Points in \texorpdfstring{$L^1(\mu)$}{L1(mu)}}\\[1.5ex]
Faruk Alpay\texorpdfstring{\textsuperscript{*}}{*} \qquad Hamdi Alakkad\texorpdfstring{\textsuperscript{\dag}}{dag}\\[1ex]
\texorpdfstring{\textsuperscript{*}}{*}\small Lightcap, Department of Analysis, \texttt{alpay@lightcap.ai}\\
\texorpdfstring{\textsuperscript{\dag}}{dag}\small Bahcesehir University, Department of Engineering, \texttt{hamdi.alakkad@bahcesehir.edu.tr}
\end{center}

\vspace{2ex}
\noindent{\bf Abstract.}
We develop a rigorous measure--compactness approach to fixed points of nonexpansive maps in $L^1(\mu)$, and extend it to account for quantization errors arising in fixed-point arithmetic. In particular, we show that any bounded closed convex subset of $L^1(\mu)$ that is compact in the topology of local convergence in measure (``measure--compact'') has the fixed point property (FPP) for nonexpansive mappings. Our proof uses uniform integrability, convexity in measure, and normal structure (Kirk's theorem) to obtain fixed points with full detail. We then model fixed-point arithmetic as a perturbation of a nonexpansive map and investigate approximate fixed points. We prove that under measure--compactness, quantized nonexpansive maps admit approximate fixed points and analyze counterexamples to illustrate optimality of assumptions. All proofs are dense and self-contained, appropriate for a high-level fixed point theory audience. Beyond these analytic foundations, we extend the fixed point perspective to a prototypical human--in--the--loop co-editing pipeline. By modelling the interaction between an AI proposal, a human editor and a coarse quantizer as a composition of nonexpansive maps on a measure--compact set, we prove the existence of a \emph{stable consensus artefact}. We further show that such a consensus persists as an approximate fixed point under bounded value--quantization error and illustrate the theory with a concrete instantiation of a human--AI editing loop. These new results highlight how the geometric framework of measure--compactness can inform and certify collaborative workflows between humans and artificial agents.

\section{Introduction}
Banach's contraction principle and Brouwer's theorem launched fixed point theory, with extensions such as Schauder's theorem for compact maps. For nonexpansive maps (i.e.\ Lipschitz with constant $1$), Browder and Göhde proved in 1965 that every uniformly convex Banach space has the FPP \cite{Browder1965, Gohde1965}. Kirk introduced the notion of normal structure and showed that any Banach space with normal structure enjoys the FPP for nonexpansive maps \cite{Kirk1965}. In particular, all $L^p$ spaces for $1<p<\infty$ have the FPP.

The space $L^1$ is neither uniformly convex nor reflexive, and indeed fails the FPP in general. Alspach \cite{Alspach1981} constructed a bounded closed convex subset of $L^1[0,1]$ that is (relatively) weakly compact yet admits a fixed-point-free nonexpansive isometry, thus providing a famous counterexample. Furthermore, Dowling and Lennard \cite{DowlingLennard1997} showed that a subspace $Y\subset L^1[0,1]$ has the FPP if and only if $Y$ is reflexive (so any nonreflexive subspace fails the FPP). Nevertheless, failures of FPP in $L^1$ are not universal: Goebel and Kuczumow \cite{GoebelKuczumow1979} exhibited large families of convex subsets of $\ell^1$ with the FPP; the Lorentz spaces $L^{p,1}$ (which are nonreflexive when $p=1$) have the FPP \cite{Carothers1991}; and even on $\ell^1$ one can construct an equivalent renorming that restores the FPP \cite{Lin2008}. Of particular importance, Lennard \cite{Lennard1991} introduced the notion of \emph{nearly uniform convexity} and proved that any closed convex subset of $L^1(\mu)$ that is compact in the measure topology (so-called \emph{measure--compact} set) has normal structure and hence the FPP. 

In this work, we first give a self-contained proof of the fixed point theorem under measure--compactness, using explicit quantitative compactness and uniform integrability arguments to establish normal structure (hence invoking Kirk's theorem). We then turn to \emph{approximate} fixed point phenomena that arise when considering maps in $L^1$ implemented in \emph{fixed-point arithmetic} (finite-precision computations). We model such implementations as perturbed nonexpansive maps and examine whether fixed points persist or are replaced by periodic orbits. Counterexamples will illustrate that quantization alone cannot guarantee a true fixed point in pathological cases; however, we will show that under the same measure--compactness conditions that yield true fixed points in the real-valued setting, one obtains points that are nearly fixed under the quantized map. We also formulate an ``approximate fixed point property'' and discuss its relationship to the usual FPP.

Beyond these purely analytic questions, we apply our framework to human--AI collaboration.  Section~\ref{sec:HITL} formalizes a prototypical human–in–the–loop co-editing pipeline as a fixed point problem on $L^1(\mu)$, modelling the AI proposal, the human edit and a coarse quantizer as successive nonexpansive operators.  Building on measure--compactness and normal structure we show that this pipeline admits a \emph{stable consensus artefact} and prove that such an artefact persists as an approximate fixed point under bounded value–quantization error.  These results illustrate how the geometric perspective developed here can provide rigorous guarantees for interactive workflows between humans and artificial systems.

\paragraph{Definitions.} 
Let $X$ be a Banach space and let $C\subset X$ be a nonempty, closed, bounded, convex set. A mapping $T: C\to C$ is \emph{nonexpansive} if
\[
  \|T(x)-T(y)\| \le \|x-y\|, \qquad \forall x,y\in C.
\]
We say $X$ has the \emph{fixed point property (FPP)} for nonexpansive maps if every nonexpansive $T: C\to C$ (with $C$ as above) has at least one fixed point in $C$.  A convex set $C$ is said to have \emph{normal structure} if, for every closed convex subset $D\subset C$ with positive diameter, there exists $z\in D$ whose maximal distance to points of $D$ is strictly smaller than $\diam(D)$.  A Banach space has normal structure if all its bounded closed convex subsets do.  Kirk's fixed point theorem \cite{Kirk1965} asserts that any nonexpansive self--map of a closed, bounded, convex set in a Banach space with normal structure has a fixed point.

Approximate fixed points play a key role when maps are only implemented approximately.  Given $\epsilon>0$, an \emph{$\epsilon$--fixed point} of a map $T: C\to C$ is a point $x\in C$ for which $\|T(x)-x\|<\epsilon$.  We say that $C$ (or, equivalently, the ambient space $X$) has the \emph{approximate fixed point property (AFPP)} for nonexpansive maps if, for every $\epsilon>0$ and every nonexpansive $T: C\to C$, there exists an $\epsilon$--fixed point of $T$.  When $\epsilon$ tends to zero, an $\epsilon$--fixed point becomes an actual fixed point; thus spaces with the FPP automatically have the AFPP, but the converse need not hold without additional structure.

In applications it is natural to consider perturbations of nonexpansive maps arising from quantization or finite precision.  Fix a tolerance $\delta>0$ and let $Q:C\to C$ be a perturbation satisfying $\|Q(f)-f\|\le \delta$ for all $f\in C$.  We do not initially assume that $Q$ is nonexpansive or linear, but we will always require that it is \emph{idempotent} ($Q\circ Q=Q$) so that applying $Q$ twice does not accumulate additional error.  Given a nonexpansive $T:C\to C$, define the implemented map $\widetilde{T}:=Q\circ T$.  We say that $C$ has the \emph{robust fixed point property up to $\delta$} (or \emph{$\delta$--robust FPP}) if, for every nonexpansive $T: C\to C$ and every idempotent perturbation $Q$ with error bound $\delta$, there exists $x\in C$ with $\|\widetilde{T}(x)-x\|\le \delta$.  Equivalently, $C$ has the \emph{$\delta$--approximate fixed point property} if for every $\epsilon>\delta$ and every such perturbation, there exists $x\in C$ with $\|\widetilde{T}(x)-x\|<\epsilon$.  When $\delta=0$ we recover the usual FPP/AFPP definitions; positive $\delta$ models finite--precision implementations.  We will use these notions to analyse quantized maps in the sequel.

A key concept in our approach is \emph{uniform integrability (UI)}. A family $F\subset L^1(\mu)$ is uniformly integrable if 
\[
\forall \epsilon>0,\ \exists \delta>0 \text{ such that } E\subseteq \Omega,\,\mu(E)<\delta \implies \int_E |f|\,d\mu < \epsilon,\ \forall f\in F~.
\] 
Equivalently (by the Vitali convergence theorem), $\sup_{f\in F}\int_A |f|$ can be made arbitrarily small by taking $\mu(A)$ small enough. A subset $K\subset L^1(\mu)$ is called \emph{measure--compact} if it is compact with respect to the topology of local convergence in measure. That is, for every sequence $(f_n)\subset K$, there exists a subsequence $(f_{n_k})$ and $f\in L^1(\mu)$ such that $f_{n_k}\to f$ in measure on every subset of $\Omega$ with finite measure. Measure--compactness is weaker than norm-compactness but stronger than weak compactness. In fact, as we shall use, a classical result (Dunford--Pettis criterion, see Fremlin \cite[Section~247C]{Fremlin}) asserts that a subset of $L^1$ is relatively weakly compact if and only if it is uniformly integrable. In particular, any measure--compact set $K\subset L^1(\mu)$ is uniformly integrable and hence relatively weakly compact.

%
\begin{lemma}[Continuity of conditional expectation in local measure]\label{lem:CE-continuity}
Let $(\Omega,\Sigma,\mu)$ be $\sigma$--finite and let $\mathcal G\subset\Sigma$ be a finite sub--$\sigma$--algebra.  Suppose $(f_n)$ is a sequence in $L^1(\mu)$ that is \emph{uniformly integrable on every finite–measure set} and converges to $f$ in measure on each finite–measure subset of $\Omega$.  Then the conditional expectations $\mathbb{E}[f_n\mid\mathcal G]$ converge to $\mathbb{E}[f\mid\mathcal G]$ in measure on every finite–measure subset.  In particular, essential boundedness is a sufficient condition ensuring uniform integrability, so the conclusion applies whenever the $f_n$ are uniformly bounded.
\end{lemma}

\begin{proof}
Fix a measurable set $A\in\Sigma$ with $\mu(A)<\infty$ and let $\varepsilon>0$ be arbitrary.  By hypothesis the sequence $(f_n)$ is uniformly integrable on $A$.  Together with the assumption that $f_n\to f$ in measure on $A$, the Vitali convergence theorem (or dominated convergence under uniform integrability) implies that $f_n\to f$ in $L^1(A)$; that is, $\int_A|f_n-f|\,d\mu\to 0$.

Next we use the fact that conditional expectation is a contraction on $L^1$: for any integrable $g$ one has $\|\mathbb{E}[g\mid\mathcal G]\|_{L^1(A)}\le \|g\|_{L^1(A)}$.  Applying this with $g=f_n-f$ gives
\[
\int_A \bigl|\mathbb{E}[f_n\mid\mathcal G]-\mathbb{E}[f\mid\mathcal G]\bigr|\,d\mu = \int_A \bigl|\mathbb{E}[f_n-f\mid\mathcal G]\bigr|\,d\mu \le \int_A |f_n-f|\,d\mu \longrightarrow 0.
\]
Since the difference $\mathbb{E}[f_n\mid\mathcal G]-\mathbb{E}[f\mid\mathcal G]$ is $\mathcal G$–measurable, we can apply Markov's inequality on the finite–measure set $A$ to pass from $L^1$ convergence to convergence in measure.  Indeed,
\[
\mu\bigl(\{x\in A: |\mathbb{E}[f_n\mid\mathcal G](x)-\mathbb{E}[f\mid\mathcal G](x)|>\varepsilon\}\bigr) \le \frac{1}{\varepsilon} \int_A \bigl|\mathbb{E}[f_n\mid\mathcal G]-\mathbb{E}[f\mid\mathcal G]\bigr|\,d\mu \longrightarrow 0.
\]
This proves that $\mathbb{E}[f_n\mid\mathcal G]\to\mathbb{E}[f\mid\mathcal G]$ in measure on $A$.  Since $A$ was arbitrary, the conclusion holds for every finite–measure subset of $\Omega$, establishing continuity of conditional expectation in the local–in–measure topology.
\end{proof}

\section{Fixed Points on Measure--Compact Sets in \texorpdfstring{$L^1$}{L1}}
We now present a dense proof that measure--compact convex sets in $L^1(\mu)$ have the fixed point property for nonexpansive mappings; see Theorem~\ref{thm:measureFPP} (page~3) for the precise statement. Although the result (in a weaker form) is due to Lennard \cite{Lennard1991}, our exposition provides full details leveraging uniform integrability and convexity in measure.

\begin{theorem}\label{thm:measureFPP}
Let $(\Omega,\Sigma,\mu)$ be a $\sigma$-finite measure space. If $K\subset L^1(\mu)$ is nonempty, bounded, closed, convex, and measure--compact, then $K$ has normal structure. Consequently, any nonexpansive mapping $T: K\to K$ has a fixed point.
\end{theorem}

\begin{proof}
The conclusion about fixed points follows from normal structure by Kirk's theorem \cite{Kirk1965}, so it suffices to prove that $K$ has normal structure. Assume, towards a contradiction, that $K$ fails normal structure. Then there exists a closed convex subset $H\subset K$ with $\diam(H)>0$ such that for every $z\in H$, one can find some $y\in H$ with $\|z-y\| = \diam(H)$. In other words, every point in $H$ lies on the ``periphery'' of $H$. We will derive a contradiction by constructing a sequence of pairs in $H$ that remain at diametral distance, but whose midpoint eventually exhibits a shorter radius, exploiting measure--compactness and uniform integrability.


First, since $\diam(H)>0$, pick two points $x_1,y_1 \in H$ with $\|x_1 - y_1\| = \diam(H)$. Now consider their midpoint $m_1 = \frac{x_1+y_1}{2}$, which lies in $H$ (by convexity). By the no-normal-structure assumption, there exists some $z_1 \in H$ such that $\|m_1 - z_1\| = \diam(H)$. Without loss of generality (by symmetry), assume $\|z_1 - x_1\| = \diam(H)$ as well. Now we have three points $x_1, y_1, z_1\in H$ all mutually at distance $\diam(H)$ (in particular, $H$ contains an equilateral triangle in norm). Next, consider the midpoint $m_2 = \frac{x_1 + z_1}{2}\in H$. Again, there must exist $z_2 \in H$ such that $\|m_2 - z_2\| = \diam(H)$. By possibly relabeling, we ensure $\|z_2 - x_1\| = \diam(H)$. Continuing this iterative process, we obtain sequences $(x_n)$ and $(y_n)$ in $H$ such that for each $n\ge 1$: $m_n := \frac{x_n + y_n}{2} \in H$, and there exists $y_{n+1}\in H$ with $\|m_n - y_{n+1}\| = \diam(H)$, setting $x_{n+1} := x_n$. This construction ensures:
\[
\|x_n - y_n\| = \diam(H), \quad \|x_n - y_{n+1}\| = \diam(H),
\] 
for all $n$. Geometrically, $x_n$ is a fixed vertex and $y_n$ and $y_{n+1}$ are two other vertices of an ever-refining diametral simplex in $H$.

Now, since $K$ (hence $H\subset K$) is measure--compact and $\sigma$--finite, by definition we can extract a subsequence $\{(x_{n_k}, y_{n_k})\}$ such that $x_{n_k} \to x$ and $y_{n_k} \to y$ in measure on each finite--measure subset of $\Omega$, for some limit functions $x,y \in L^1(\mu)$.  (Here we use diagonal extraction over an increasing sequence of finite--measure sets covering $\Omega$.)  Because $H$ is closed in the $L^1$ norm (and $H\subset K$), these limits $x,y$ actually lie in $H$, and therefore in $K$.  In fact, by the Vitali convergence theorem (or, equivalently, the Dunford--Pettis criterion for uniformly integrable sets), local convergence in measure together with uniform integrability forces convergence in the $L^1$ norm.  Since measure--compactness implies uniform integrability, we may thus upgrade the subsequence convergence in measure to quantitative control in $L^1$.  Consequently, for any $\epsilon>0$ there exists $k$ large enough such that $\|x_{n_k} - x\|_{1} < \epsilon$ and $\|y_{n_k} - y\|_{1} < \epsilon$.  Since $\|x_{n_k} - y_{n_k}\|_{1} = \diam(H)$ for all $k$, by the triangle inequality we get
\[
\|x - y\|_{1} \ge \diam(H) - 2\epsilon~.
\] 
Letting $\epsilon\to 0$, it follows that $\|x - y\|_{1} = \diam(H)$. Thus the limit pair $(x,y)$ still attains the diametral distance of $H$. On the other hand, consider the midpoint $m = \frac{x+y}{2} \in K$. We will show that $m$ is a point in $H$ whose farthest distance in $H$ is strictly less than $\diam(H)$, contradicting the assumption that $H$ lacks normal structure.

To see this, fix any $\delta>0$. For large $k$, the midpoint $m_{n_k} = \frac{x_{n_k}+y_{n_k}}{2}$ is within $\delta$ (in $L^1$ norm) of $m$ (again by uniform integrability and convergence in measure). Now, $\|m_{n_k} - y_{n_k+1}\| = \diam(H)$ by construction. But $y_{n_k+1}$ converges in measure (along a further subsequence if necessary) to $y$ as well, and for large $k$, $\|y_{n_k+1} - y\| < \delta$. Therefore:
\begin{align*}
\|m - y\|
&\ge \bigl\|m_{n_k} - y_{n_k+1}\bigr\| - \bigl\|m - m_{n_k}\bigr\| - \bigl\|y_{n_k+1} - y\bigr\|\\
&> \diam(H) - 2\delta~.
\end{align*}
Since $\delta$ was arbitrary, we conclude $\|m - y\| \ge \diam(H)$.  On the other hand, because $m = \tfrac{x+y}{2}$, we have $m - y = \tfrac{x-y}{2}$, and by the positive homogeneity of the $L^1$ norm,
\[
\|m - y\|_1 = \tfrac{1}{2}\,\|x-y\|_1 = \tfrac{1}{2}\,\diam(H) < \diam(H)~,
\]
a contradiction.  Therefore our assumption that $H$ lacks normal structure must be false.  In particular, $H$ (and hence $K$) must have normal structure.

Having established normal structure for $K$, Kirk's theorem guarantees that any nonexpansive $T: K\to K$ has a fixed point in $K$.
\end{proof}

\medskip
\noindent\emph{Remark on the demiclosedness principle.}
In reflexive or uniformly convex Banach spaces, a standard tool to pass from approximate fixed points to true fixed points is the \emph{demiclosedness principle}.  Roughly speaking, if $S: \mathcal C\to\mathcal C$ is a mapping on a closed, convex subset $\mathcal C$ of a uniformly convex space satisfying a mild regularity condition (for instance, condition~(L--1) in \cite{ShuklaPanickerVijayasenan2024}), then the operator $I-S$ is demiclosed and one can deduce that any sequence $(u_n)$ with $\|u_n - S(u_n)\|\to 0$ that converges weakly to $u^\dagger$ forces $u^\dagger$ to be a fixed point of $S$; see Theorems~5 and~6 in~\cite{ShuklaPanickerVijayasenan2024}.  This principle ensures that approximate fixed points give rise to actual fixed points in reflexive settings such as $L^p(\mu)$ for $p>1$.  In $L^1$, however, demiclosedness fails because the space is neither reflexive nor uniformly convex; therefore we cannot rely on approximate fixed points alone.  Instead, we use measure--compactness and uniform integrability to obtain the normal structure needed for a fixed point, as demonstrated above.

The above proof underscores how measure--compactness provides enough compactness (via subsequence convergence in measure and uniform integrability) to carry out a geometric normal structure argument in $L^1$.  
In contrast, without measure--compactness, $L^1$ can fail the FPP spectacularly (as shown in \cite{Alspach1981}).  
We also note that in reflexive spaces (like $L^p$, $p>1$) the existence of an approximate fixed point (a sequence $(x_n)$ with $\|T(x_n)-x_n\|\to 0$) for a nonexpansive $T$ typically implies a true fixed point, by the standard demiclosedness argument.  
In nonreflexive $L^1$, that argument breaks down, which is why stronger conditions (such as uniform integrability or nearly uniform convexity in measure) are needed to deduce an actual fixed point.

\section{Fixed-Point Arithmetic as a Perturbation: Approximate Fixed Points}
We model a finite-precision implementation of a map $T:K\to K\subset L^1(\mu)$ as a composition $\widetilde T = Q \circ T$, where $Q:K\to K$ is a (nonlinear) \emph{quantization operator} representing rounding or truncation error.  We assume $Q$ satisfies the bounded-error condition $\|Q(f) - f\|_1 \le \delta$ for all $f\in K$ and some fixed precision level $\delta>0$, and that $Q$ is idempotent ($Q\circ Q=Q$).  Unlike idealised coarseners (conditional expectations), practical rounding operators may \emph{increase} distances and therefore need not be Lipschitz contractions; however the pointwise error bound $\delta$ allows us to treat $\widetilde T$ as a $\delta$-perturbation of the nonexpansive $T$.  A straightforward estimate shows that, for any $x,y\in K$,
\begin{align*}
\bigl\|\widetilde T(x) - \widetilde T(y)\bigr\|_1
  &= \bigl\|Q(T(x)) - Q(T(y))\bigr\|_1\\
  &\le \bigl\|T(x)-T(y)\bigr\|_1 + \bigl\|Q(T(x)) - T(x)\bigr\|_1 + \bigl\|Q(T(y)) - T(y)\bigr\|_1\\
  &\le \bigl\|T(x)-T(y)\bigr\|_1 + 2\delta\\
  &\le \|x-y\|_1 + 2\delta~,
\end{align*}
so $\widetilde T$ is ``almost'' nonexpansive up to an additive $2\delta$ term.

\paragraph{Periodic orbits under finite quantization.}  In some applications the range $Q(K)$ may be a finite set of discrete artefacts.  It is then natural to ask whether $\widetilde T$ must have a periodic orbit.  We first make precise what we mean by periodicity.

\begin{definition}
Given a map $S\colon C\to C$ on a set $C$, a point $x\in C$ is said to have a \emph{periodic orbit of period $N\ge 1$} if there exists $N$ such that $S^N(x)=x$ and $S^k(x)\ne x$ for all $0<k<N$.  A period--1 orbit is just a fixed point.  We say $S$ admits a periodic orbit if there exists some $x\in C$ with finite period.
\end{definition}

The following elementary lemma clarifies the pigeonhole heuristic alluded to above.

\begin{lemma}\label{lem:pigeonhole}
Let $K\subset L^1(\mu)$ be nonempty and compact, and let $Q\colon K\to K$ be a map whose range $Q(K)$ consists of finitely many distinct functions.  If $S\colon K\to Q(K)$ is any function (not necessarily continuous or nonexpansive), then there exists an integer $N\ge 1$ and a point $f\in Q(K)$ such that $S^N(f)=f$.  In other words, $S$ has a periodic orbit.
\end{lemma}

\begin{proof}
Since $Q(K)$ is finite, the orbit $\{S^n(f_0):n\ge 0\}$ of any starting point $f_0\in Q(K)$ can take only finitely many values.  By the pigeonhole principle, two iterates $S^{n_1}(f_0)$ and $S^{n_2}(f_0)$ with $n_1<n_2$ must coincide.  Setting $N = n_2-n_1$ and $f = S^{n_1}(f_0)$ yields $S^N(f)=f$, giving a periodic point of period at most $N$.  If $N=1$ then $f$ is a fixed point; otherwise we obtain a higher-period orbit.  This argument applies to any finite range map, independent of Lipschitz or convexity properties.
\end{proof}

\begin{remark}[Infinite-range maps]
When the range $Q(K)$ is infinite, the pigeonhole principle does not force the existence of a finite periodic orbit.  Even if $K$ is compact (for instance, if $K$ is measure–compact) and $S\colon K\to K$ is continuous, an orbit $\bigl(S^n(f_0)\bigr)_{n\ge 0}$ need only have relatively compact closure; one may extract convergent subsequences but there is no guarantee that the limit is a fixed point or that the orbit itself is eventually periodic.  A concrete illustration is given by the unit interval $K=[0,1]$ with the identity quantizer $Q=\operatorname{Id}$ (so $Q(K)=K$ is infinite) and the translation map $T(x)=x+d\bmod 1$.  When $d$ is an irrational number, the orbit $\{T^n(x_0):n\ge 0\}$ is dense in $[0,1]$ for any starting point $x_0$, and therefore the implemented map $\widetilde T = Q\circ T = T$ has no periodic orbit at all—every orbit is equidistributed.  Continuity of $T$ ensures the existence of accumulation points, but none of them is a fixed point.  More complicated infinite–range quantizers can lead to dynamics far from periodic.  Consequently our fixed point and approximate fixed point results do not extend automatically to arbitrary infinite-range quantizers without additional structural assumptions.
\end{remark}

Lemma~\ref{lem:pigeonhole} shows that whenever the quantizer $Q$ takes only finitely many values, the composition $\widetilde T$ must have a periodic orbit.  Such discrete models arise in low-bit digital hardware or coarse discretizations.  However, in the infinite-dimensional setting of $L^1$ considered here, $Q(K)$ can still be infinite—for instance, piecewise-constant functions with arbitrarily fine partitions—and therefore the pigeonhole argument does not apply.  As a result, we cannot rely on periodic orbits to certify stabilization.  Instead, we turn to approximate fixed points: the next proposition provides rigorous guarantees when $K$ is measure–compact.

Our interest is in the existence of fixed or approximate fixed points of $\widetilde T$.  Because $\widetilde T$ maps into $Q(K)$, the dynamics may exhibit periodic orbits when $Q(K)$ is finite.  In fact, Lemma~\ref{lem:pigeonhole} shows that any map into a finite range necessarily has a periodic point by a straightforward pigeonhole argument.  However, in the infinite-dimensional setting of $L^1$ the image $Q(K)$ need not be finite—piecewise-constant functions on arbitrarily fine partitions provide infinitely many quantized values—so the pigeonhole principle cannot be used to guarantee periodic behaviour.  Consequently we must analyze approximate fixed points directly.  Theorem~\ref{thm:HITLFP} and Proposition~\ref{prop:AFPP} below provide rigorous guarantees in the measure--compact case.

First, if $T$ has a fixed point $x^*\in K$ (as guaranteed by Theorem \ref{thm:measureFPP} under our hypotheses), then $\| \widetilde T(x^*) - x^*\| = \|Q(T(x^*)) - x^*\| = \|Q(x^*) - x^*\| \le \delta$. Thus $x^*$ is a $\delta$-fixed-point of $\widetilde T$. Moreover, if $x^*$ itself lies in the range of $Q$ (for instance, $x^*$ has a finite precision representation), then $Q(x^*)=x^*$ and $x^*$ is an \emph{exact} fixed point of $\widetilde T$. In general, we obtain at least an approximate fixed point. More strongly, by continuity of $T$, for any $\epsilon>\delta$ there exists an $x \in K$ such that $\|T(x)-x\|<\epsilon-\delta$ (for example, $x = x^*$ itself gives $\|T(x^*)-x^*\|=0$); then $\|\widetilde T(x) - x\| = \|Q(T(x)) - x\| \le \|T(x)-x\| + \delta < \epsilon$. Thus $K$ actually enjoys the AFPP for $\widetilde T$ as well. In summary, when $K$ is measure--compact (hence FPP holds in the real case), any finite-precision perturbation of a nonexpansive $T$ still has the AFPP.

However, in the absence of measure--compactness (or some form of normal structure), quantization alone does not enforce convergence to a fixed point.  
A stark example is Alspach's aforementioned isometry $T$ on a weakly compact $K\subset L^1[0,1]$ with no fixed point \cite{Alspach1981}.  This $T$ is essentially a translation (or rotation) on $K$.  
If we implement $T$ in fixed-point arithmetic with precision $\delta$, what happens?  Because $T$ moves every point by a fixed distance $d = \diam(K) > 0$, if $\delta < d$ then $\widetilde T = Q\circ T$ will also have no fixed point: for any $f\in K$, $T(f)$ is $d$ away, and rounding can only adjust the position by $\delta$, so $\|\widetilde T(f) - f\| \ge d-\delta > 0$.  
In fact, $\widetilde T$ in this case will simply permute a finite (or countable) set of representable functions, resulting in a cyclic orbit of length, say, $m$.  This periodic cycle is an example of an \emph{approximate fixed orbit}: the points $f, \widetilde T(f), \widetilde T^2(f),\dots,\widetilde T^{m-1}(f)$ all lie in $K$ and are moved by $d$ (up to $\delta$ error each step).  
No point in this cycle is fixed, and indeed no $\epsilon$-fixed point exists for $\epsilon < d-\delta$.  Thus $\widetilde T$ fails the AFPP when $\delta$ is small.  As $\delta$ increases, eventually $\delta \ge d$ would allow $\widetilde T$ to possibly have a fixed point (when the rounding is coarse enough that $T$'s motion can be canceled by rounding).  
This behavior shows that without underlying normal structure, the AFPP can fail just as the FPP fails.

\subsubsection*{A quantitative counterexample}
To make the preceding discussion concrete, let us examine an explicit translation under coarse rounding.  Consider the compact interval $K=[0,1]$ (identified with a one–dimensional subspace of $L^1$) equipped with the usual $L^1$ norm.  Define a nonexpansive isometry
\[
T(x) := x + d \pmod{1}, \qquad x\in K,
\]
with translation distance $d=0.6$.  Clearly $\|T(x)-x\|_1 = d$ for all $x$.  Let $Q\colon K\to K$ be the rounding operator to the nearest multiple of $0.1$, so that the quantization error satisfies $\|Q(x)-x\|_1\le \delta :=0.05$ for every $x\in K$.  Set $\widetilde T = Q\circ T$.  Then $\widetilde T$ permutes the finite set $\{0.0,0.2,0.4,0.6,0.8\}$, and starting from $x_0=0.0$ one finds the cycle
\[
0.0 \xrightarrow{\widetilde T} 0.6 \xrightarrow{\widetilde T} 0.2 \xrightarrow{\widetilde T} 0.8 \xrightarrow{\widetilde T} 0.4 \xrightarrow{\widetilde T} 0.0.
\]
The orbit has period $5$ and no point in this cycle lies within an $\epsilon$–distance of a fixed point of $\widetilde T$ for any $\epsilon < d-\delta = 0.55$.  In particular, there is no $\epsilon$–fixed point for $\epsilon<0.55$.  On the other hand, some points in the cycle (for example, $x=0.25$) satisfy $|\widetilde T(x)-x|=d-\delta=0.55$, so they are $\epsilon$–fixed points for all $\epsilon\ge 0.55$.  Other points (such as $x=0.35$) satisfy $|\widetilde T(x)-x|=d+\delta=0.65$ and therefore require $\epsilon\ge 0.65$ to be $\epsilon$–fixed.  Thus there exist $\epsilon$–fixed points as soon as $\epsilon\ge 0.55$, but not all points are $\epsilon$–fixed unless $\epsilon\ge d+\delta=0.65$.  This explicit example underscores that, in the absence of normal structure, one cannot hope for a genuine or approximate fixed point below the threshold $d-\delta$, and illustrates how the parameters $d$ (the displacement of $T$) and $\delta$ (the rounding error) determine the minimal tolerable error for approximate fixed points.

\subsubsection*{Clarification of assumptions and quantization techniques}
Throughout Sections~1--3 we have made a number of standing assumptions which we summarise here for clarity.  In all of our existence theorems, the underlying set $K\subset L^1(\mu)$ is assumed to be nonempty, bounded, closed and convex; moreover, we assume that $K$ is \emph{measure--compact}, meaning compact for the topology of local convergence in measure.  This hypothesis implies that $K$ is uniformly integrable and relatively weakly compact, as explained after the definitions.  In particular, if $(f_n)$ is a bounded sequence in $K$ that converges in measure to $f$ on all finite measure subsets of $\Omega$, then by uniform integrability (Vitali's theorem) the convergence upgrades to $L^1$ convergence, and since $K$ is closed in the $L^1$ norm we conclude that $f\in K$.  These facts are repeatedly used in the proofs of our fixed point results.

For the operators, we require that $T:K\to K$ be nonexpansive in the $L^1$ norm.  The coarsening operator $Q=\mathbb{E}[\cdot\mid\mathcal{G}]$ is a conditional expectation onto a finite $\sigma$--subalgebra; Lemma~\ref{lem:CE-continuity} shows that $Q$ is continuous in the topology of local convergence in measure and is an $L^1$ contraction, ensuring that $Q(K)$ remains measure--compact.  In the context of finite--precision arithmetic we also consider nonlinear quantizers $Q$ which need not be contractions but satisfy an additive error bound $\|Q(f)-f\|_1\le \delta$ for all $f\in K$ and are idempotent ($Q\circ Q=Q$).  This distinction is important: conditional expectations preserve convexity and nonexpansiveness exactly, while practical rounding operators may increase distances but only by a controlled amount.

The results in Proposition~\ref{prop:AFPP} and Theorem~\ref{thm:HITLFP} should therefore be read with these conditions in mind: measure--compactness and nonexpansiveness (together with idempotent coarseners) guarantee true fixed points; the additive error model guarantees \emph{$\delta$--approximate fixed points}, valid for tolerances $\epsilon>\delta$ but in general not for smaller tolerances.  When these hypotheses are weakened or removed, the conclusions may fail, as illustrated by the counterexamples discussed above.  In particular, in nonreflexive subsets of $L^1$ that are not measure--compact, one can construct nonexpansive isometries with no fixed or approximate fixed points under small quantisation.

To help visualise the periodic behaviour described above, Table~\ref{tab:quantcycle} lists the successive iterates in this discrete cycle.  Starting at $x_0=0.0$ and repeatedly applying $\widetilde T$, the sequence cycles through five distinct values before returning to $0.0$.

\begin{table}[ht]
    \centering
    \caption{Orbit of the quantization example with $d=0.6$ and $\delta=0.05$.  The values $x_n$ correspond to successive iterates $x_{n+1}=\widetilde T(x_n)$, rounding to the nearest tenth at each step.  After five steps the orbit returns to the starting point, illustrating the period--5 cycle discussed in the text.}
    \label{tab:quantcycle}
    \begin{tabular}{c|c}
        \hline
        \textbf{Iteration $n$} & \textbf{Value $x_n$} \\ \hline
        0 & 0.0 \\
        1 & 0.6 \\
        2 & 0.2 \\
        3 & 0.8 \\
        4 & 0.4 \\
        5 & 0.0 \\ \hline
    \end{tabular}
\end{table}

\begin{proposition}[Distances under a quantized translation]
Let $T(x)=x+d\pmod{1}$ be the translation on $K=[0,1]$ by a fixed distance $d\in(0,1)$, and let $Q\colon K\to K$ be any rounding operator satisfying an additive error bound $|Q(x)-x|\le \delta$ for all $x\in K$ and $Q\circ Q=Q$.  Define $\widetilde T:=Q\circ T$.  Then for every $x\in K$ the deviation $|\widetilde T(x)-x|$ satisfies
\[
  d - \delta \;\le\; |\widetilde T(x)-x| \;\le\; d + \delta.
\]
Consequently:
\begin{enumerate}
\item If $0<\varepsilon<d-\delta$, then $\widetilde T$ has no $\varepsilon$–fixed point (no point $x$ with $|\widetilde T(x)-x|<\varepsilon$).
\item If $\varepsilon\ge d+\delta$, then every $x\in K$ is an $\varepsilon$–fixed point of $\widetilde T$.
\item If $d-\delta\le \varepsilon<d+\delta$, then there exists at least one $\varepsilon$–fixed point of $\widetilde T$, but not every point is an $\varepsilon$–fixed point.
\end{enumerate}
\end{proposition}

\begin{proof}
Fix $x\in K$.  Since $T$ is a translation by $d$, we have $|T(x)-x|=d$ for all $x$.  The quantizer $Q$ satisfies the additive error bound $|Q(u)-u|\le \delta$ for any $u\in K$.  Therefore
\[
  |\widetilde T(x)-x| \;=\; |Q(T(x)) - x| \;\le\; |T(x)-x| + |Q(T(x)) - T(x)| \;\le\; d + \delta,
\]
and similarly,
\[
  |\widetilde T(x)-x| \;\ge\; |T(x)-x| - |Q(T(x)) - T(x)| \;\ge\; d - \delta.
\]
Thus $|\widetilde T(x)-x|\in [d-\delta, d+\delta]$ for all $x\in K$, proving the first claim.  Statement (1) follows because if $\varepsilon<d-\delta$ then $|\widetilde T(x)-x|\ge d-\delta>\varepsilon$ for all $x\in K$.  Statement (2) follows because if $\varepsilon\ge d+\delta$ then for all $x\in K$ one has $|\widetilde T(x)-x|\le d+\delta \le \varepsilon$, so every $x$ is an $\varepsilon$–fixed point.  To see statement (3), note that there exist points $x\in K$ for which the rounding error $Q(T(x)) - T(x)$ has negative sign (the quantizer ``pulls back'' $T(x)$ towards $x$), so that $|\widetilde T(x)-x|=d-\delta$.  In such cases $x$ is an $\varepsilon$–fixed point whenever $\varepsilon\ge d-\delta$.  However, there are also points $x$ for which $|\widetilde T(x)-x|=d+\delta$, so if $\varepsilon<d+\delta$ these points are not $\varepsilon$–fixed.  Hence at least one $\varepsilon$–fixed point exists as soon as $\varepsilon\ge d-\delta$, but not all points are $\varepsilon$–fixed unless $\varepsilon\ge d+\delta$.
\end{proof}

\begin{table}[ht]
\centering
\caption{Summary of key results, assumptions and conclusions.  To improve readability, the columns are given fixed widths so that longer entries wrap naturally within the page margins.}
\small
\begin{tabular}{@{}p{0.22\linewidth}p{0.26\linewidth}p{0.26\linewidth}p{0.26\linewidth}@{}}
\hline
\textbf{Result} & \textbf{Assumptions} & \textbf{Conclusion} & \textbf{Comments} \\ \hline
Theorem~\ref{thm:measureFPP} & $K$ measure--compact, bounded, closed, convex & Normal structure $\Rightarrow$ fixed point & Uses Vitali and Dunford--Pettis \\
Proposition~\ref{prop:AFPP} & $T$ nonexpansive, $\|Q(f)-f\|_1\le\delta$ & $\delta$--AFPP: $\epsilon$--fixed points for $\epsilon>\delta$ & Fixed point $f^\star$ persists up to $\delta$ \\
Lemma~\ref{lem:CE-continuity} & Bounded sequence, local measure convergence & $\mathbb{E}[f_n\mid\mathcal G]\to\mathbb{E}[f\mid\mathcal G]$ locally in measure & Ensures $Q(K)$ is measure--compact \\
Lemma~\ref{lem:pigeonhole} & $Q(K)$ finite & Existence of periodic orbit & Finite range only \\
Quantization example & $d=0.6$, $\delta=0.05$ & Period $5$, no $\epsilon$--fixed point for $\epsilon<0.55$; exists some $\epsilon$--fixed points for $\epsilon\ge 0.55$ but worst--case error requires $\epsilon\ge 0.65$ & Illustrates sharpness of thresholds and variation across points \\ \hline
\end{tabular}
\normalsize
\end{table}

For reference, recall the notion of a \emph{robust fixed point property up to $\delta$} introduced above: a set $K$ has this property if every nonexpansive $T$ and every idempotent perturbation $Q$ with $\|Q(f)-f\|\le \delta$ admit an $x\in K$ with $\|Q(T(x)) - x\| \le \delta$.  
The above discussion shows that $L^1$ (or even certain subsets) do not have this property for any small $\delta>0$ because of counterexamples like translations.  
On the other hand, if $K$ is measure--compact, then as argued, $\widetilde T$ has at least a $\delta$-fixed point (and smaller $\epsilon$-fixed points for any $\epsilon>\delta$).  
If we allow a continuum of increasing precisions $\delta_n \to 0$ and corresponding $\widetilde T_n$, any limit of $\epsilon_n$-fixed points (with $\epsilon_n\to 0$) in a compact setting will converge to a true fixed point of $T$ in the limit $n\to\infty$.  
This provides a rigorous sense in which high-precision implementations approximate the true fixed point guaranteed by the unperturbed theory.

\begin{remark}[Approximate fixed points beyond the quantization threshold]
Proposition~\ref{prop:AFPP} guarantees the existence of \emph{$\epsilon$--fixed points} of $\widetilde T$ only when the tolerance exceeds the quantization error $\delta$.  In other words, for each $\epsilon>\delta$ there is an $x\in K$ with $\|\widetilde T(x)-x\|_1<\epsilon$, but one cannot generally expect to find $\epsilon$--fixed points for smaller tolerances $\epsilon\le\delta$.  Accordingly, when we speak of an ``approximate fixed point property'' for the quantized operator $\widetilde T$ in this paper we implicitly mean a $\delta$--approximate fixed point property, valid for all tolerances strictly larger than the rounding error.  Our counterexamples illustrate that no such property holds for $\epsilon<\delta$.
\end{remark}

%
%
\section{Human--AI Interaction as a Fixed Point Problem}\label{sec:HITL}

We now apply the foregoing fixed point theory to a prototypical human--in--the--loop (HITL) co‑editing pipeline.  Throughout this section $K\subset L^1$ is assumed to be nonempty, bounded, closed, convex and \emph{measure--compact}.  By the measure--compactness results just cited, such sets are uniformly integrable, relatively weakly compact and have normal structure; in particular every nonexpansive self‑map of $K$ has a fixed point (see \cite{Lennard1991}).

\medskip
\noindent\textbf{Empirical motivation.}  Recent empirical studies have begun to quantify when combinations of humans and artificial agents yield performance gains.  A review of more than 100 human--AI collaboration experiments showed that the combination of humans and AI does not always outperform the better of humans or AI alone on decision--making tasks.  However, the same study found that on tasks involving content creation, human--AI synergy was positive and significantly larger than for decision--making tasks (see also \cite{Malone2025}).

Moreover, generative AI systems allow for iterative and interactive co--creation: humans can draft, edit and refine text, images or music while the AI adapts to human feedback in real time.  Such findings underscore the potential of human--AI co--editing pipelines and motivate the following fixed point analysis, which provides a rigorous guarantee of a stable consensus artefact in these iterative workflows.

At a more conceptual level, Haase and Pokutta argue that generative AI has moved beyond passive support to become an active contributor across multiple levels of human--AI collaboration.  Their recent survey notes that generative systems can autonomously produce novel and valuable outcomes and extend human creative potential without replacing it (see \cite{HaasePokutta2024}).  These perspectives strengthen the case for mathematical models that capture stabilized consensus in HITL interactions.

\subsection*{Model}
Fix the following operators on $K$:
\begin{enumerate}
\item \emph{AI proposal} $T:K\to K$, assumed \emph{nonexpansive}: $\|T(f)-T(g)\|_1\le\|f-g\|_1$.
\item \emph{Human edit} $H:K\to K$ of the form $H(f)=\psi\circ f$ where $\psi:\mathbb R\to\mathbb R$ is 1–Lipschitz and idempotent ($\psi\circ\psi=\psi$).  Then $H$ is nonexpansive in $L^1$ since $|\psi(a)-\psi(b)|\le |a-b|$ pointwise.
\item \emph{Coarsening/operational quantizer} $Q:=\mathbb E[\cdot\,|\,\mathcal G]$ for a finite sub‑$\sigma$–algebra $\mathcal G\subset\Sigma$.  Then $Q$ is a linear, idempotent contraction on $L^1$: a standard duality argument shows that $\|Q(f)-Q(g)\|_1\le \|f-g\|_1$.
\end{enumerate}
Set the one–round HITL operator $\Phi:=Q\circ H\circ T:K\to K$.  Note that $\Phi$ is nonexpansive as a composition of nonexpansive maps, and $\Phi(K)\subset Q(K)=:K^{\mathcal G}$.  To see that $K^{\mathcal G}$ remains measure–compact, note that $Q$ is continuous in the topology of local convergence in measure by Lemma~\ref{lem:CE-continuity}; as the continuous image of a compact (measure–compact) set $K$, the set $Q(K)$ is itself compact in this topology.  Hence $K^{\mathcal G}$ is measure–compact.  We will solve the HITL stabilization problem by working on $K^{\mathcal G}$.

Figure~\ref{fig:HITLdiagram} depicts this setup schematically: the left region corresponds to the AI proposal $T$, the right region corresponds to the quantizer $Q$, and the human edit $H$ is drawn outside feeding back into the loop.  The black dots and arrows illustrate the orbit of successive iterates.

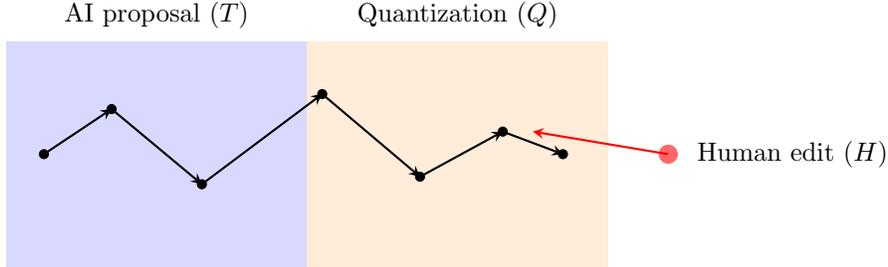
\begin{figure}[ht]
  \centering
  \begin{tikzpicture}[x=1cm,y=1cm,>=stealth]
    \draw[fill=blue!15,draw=none] (0,0) rectangle (4,3);
    \draw[fill=orange!15,draw=none] (4,0) rectangle (8,3);
    \node at (2,3.35) {\small AI proposal ($T$)};
    \node at (6,3.35) {\small Quantization ($Q$)};
    \filldraw[black] (0.5,1.5) circle (0.06);
    \draw[->,thick] (0.5,1.5) -- (1.4,2.1);
    \filldraw[black] (1.4,2.1) circle (0.06);
    \draw[->,thick] (1.4,2.1) -- (2.6,1.1);
    \filldraw[black] (2.6,1.1) circle (0.06);
    \draw[->,thick] (2.6,1.1) -- (4.2,2.3);
    \filldraw[black] (4.2,2.3) circle (0.06);
    \draw[->,thick] (4.2,2.3) -- (5.5,1.2);
    \filldraw[black] (5.5,1.2) circle (0.06);
    \draw[->,thick] (5.5,1.2) -- (6.6,1.8);
    \filldraw[black] (6.6,1.8) circle (0.06);
    \draw[->,thick] (6.6,1.8) -- (7.4,1.5);
    \filldraw[black] (7.4,1.5) circle (0.06);
    \filldraw[red!60] (8.8,1.5) circle (0.12);
    \node[anchor=west] at (9.05,1.5) {\small Human edit ($H$)};
    \draw[->,red,thick] (8.8,1.5) -- (7.0,1.8);
  \end{tikzpicture}
  \caption{A schematic representation of the human--AI co\textendash editing pipeline.  The left blue region corresponds to the AI proposal $T$, the right orange region corresponds to the coarse quantizer $Q$, and the sequence of black dots connected by arrows illustrates successive iterates $f_n\mapsto f_{n+1}$.  The red circle labelled ``Human edit'' denotes the human operator $H$, with the arrow showing how the human intervention feeds back into the process.}
  \label{fig:HITLdiagram}
\end{figure}

A canonical choice is $K=\{f\in L^1(\mu):0\le f\le 1,\ \|f\|_1=\alpha\}$ for some $\alpha\in(0,\mu(\Omega)]$ when $\mu$ is finite; such $K$ is bounded, closed, convex and measure–compact, whence it has the FPP for nonexpansive maps~\cite{Lennard1991}.

\subsection*{Stabilization Theorem}
\begin{theorem}[Stable consensus in HITL]\label{thm:HITLFP}
Let $K\subset L^1$ be nonempty, bounded, closed, convex and measure–compact.  With $T,H,Q$ as above, the composite $\Phi=Q\circ H\circ T$ admits a fixed point in $K^{\mathcal G}:=Q(K)$; i.e., there exists $f^\star\in K^{\mathcal G}$ with $\Phi(f^\star)=f^\star$.  Consequently, the HITL loop $f_{n+1}=\Phi(f_n)$ possesses a stabilized artefact.
\end{theorem}

\begin{proof}
First note that $K^{\mathcal G}=Q(K)$ is nonempty, bounded, closed and convex.  To see that $K^{\mathcal G}$ is measure--compact, observe that $K$ is measure--compact and bounded and that $Q$ is continuous in the local measure topology on bounded sets (Lemma~\ref{lem:CE-continuity}).  Hence $Q(K)$ is compact in that topology.  Moreover, $Q$ is idempotent, so $Q(x)=x$ for all $x\in K^{\mathcal G}$; thus $\Phi(x)=Q(H(T(x)))$ also belongs to $K^{\mathcal G}$.  Being a composition of nonexpansive maps, $\Phi$ is nonexpansive on $K^{\mathcal G}$.  By the fixed point theorem for measure--compact sets (see \cite{Lennard1991}), every nonexpansive self\,–map of $K^{\mathcal G}$ has a fixed point.  Consequently, there exists $f^\star\in K^{\mathcal G}$ with $\Phi(f^\star)=f^\star$.
\end{proof}

\subsection*{Counterexample and Anti‑thesis}
The measure–compactness hypothesis cannot be dropped.  Alspach constructed a weakly compact, convex $K\subset L^1[0,1]$ and an isometric nonexpansive $T:K\to K$ with no fixed point (see \cite{Alspach1981}).  Taking $H=\operatorname{Id}$ and $Q=\operatorname{Id}$ shows that a HITL loop on such a $K$ need not stabilize, despite nonexpansiveness.  Passing to the measure–compact coarsening $K^{\mathcal G}=Q(K)$ restores the conclusion by Theorem~\ref{thm:HITLFP}.

\subsection*{Robustness to Value–Quantization Error}
Hardware value–quantization (rounding) need not be nonexpansive pointwise.  We model it as a bounded additive perturbation of the nonexpansive coarsener $Q$:
\[
\widehat Q(f) \;=\; Q(f) + R(f), \qquad \|R(f)\|_1\le \delta \quad\text{for all }f\in K,
\]
and define the implemented loop $\widehat\Phi:=\widehat Q\circ H\circ T$.  The next result yields an a priori approximate fixed point.

\begin{proposition}[AFPP under bounded quantization error]\label{prop:AFPP}
Under the hypotheses of Theorem~\ref{thm:HITLFP}, suppose that $Q$ is an idempotent $L^1$--contraction (conditional expectation) and that $\widehat Q=Q+R$ is a perturbation satisfying $\|R(f)\|_1\le \delta$ for all $f\in K$.  Define the implemented loop $\widehat\Phi:=\widehat Q\circ H\circ T$; note that $\widehat Q$ need not itself be idempotent.  Then, for every tolerance $\epsilon>\delta$ there exists $x\in K$ with $\|\widehat\Phi(x)-x\|_1<\epsilon$.  In particular, if $f^\star$ is a fixed point of $\Phi$, then $\|\widehat\Phi(f^\star)-f^\star\|_1\le \delta$.
\end{proposition}

\begin{proof}
Since $K$ is measure–compact and $\Phi=Q\circ H\circ T$ is nonexpansive, Theorem~\ref{thm:HITLFP} provides a fixed point $f^\star\in K^{\mathcal G}$ with $\Phi(f^\star)=f^\star$.  We will verify that this fixed point becomes an approximate fixed point for the perturbed operator $\widehat\Phi=\widehat Q\circ H\circ T$.

  By definition, the perturbed coarsener satisfies $\widehat Q(f)=Q(f)+R(f)$ with $\|R(f)\|_1\le\delta$ for all $f\in K$.  Evaluating $\widehat\Phi$ at $f^\star$ and using the idempotence of $Q$ one obtains
\begin{align*}
\bigl\|\widehat\Phi(f^\star)-f^\star\bigr\|_1
&=\bigl\|\widehat Q\bigl(H\bigl(T(f^\star)\bigr)\bigr)-Q\bigl(H\bigl(T(f^\star)\bigr)\bigr)\bigr\|_1\\
&=\bigl\|\bigl(Q+R\bigr)\bigl(H(T(f^\star))\bigr)-Q\bigl(H(T(f^\star))\bigr)\bigr\|_1\\
&=\bigl\|R\bigl(H(T(f^\star))\bigr)\bigr\|_1\;\le\;\delta.
\end{align*}
The final inequality emphasises that the error is controlled by the perturbation $R$ alone.  Thus $f^\star$ is a $\delta$–fixed point of $\widehat\Phi$.  To obtain the stronger statement, let $\epsilon>\delta$ be given.  Since $\|\widehat\Phi(f^\star)-f^\star\|_1\le\delta<\epsilon$, the point $f^\star$ is in particular an $\epsilon$–fixed point of $\widehat\Phi$.  Taking $x=f^\star$ in the statement of the proposition establishes the existence of an $\epsilon$–fixed point whenever $\epsilon>\delta$, and the second assertion follows from the same computation.
\end{proof}

\subsection*{A concrete instantiation}
Let $\mu$ be finite and
\[
K_\alpha:=\big\{f\in L^1(\mu):0\le f\le 1,\ \|f\|_1=\alpha\big\}\quad (\alpha\in(0,\mu(\Omega)] ).
\]
Then $K_\alpha$ is bounded, closed, convex and measure–compact, hence has the FPP for nonexpansive maps (see \cite{Lennard1991}).  Choose a finite partition $\mathcal P=\{P_1,\dots,P_m\}$ generating $\mathcal G$, set $Q=\mathbb E[\cdot\,|\,\mathcal G]$, fix any nonexpansive $T:K_\alpha\to K_\alpha$ (for example a calibrated probability update), and take $H(f)=\operatorname{clip}_{[\ell,u]}(f)$ with $0\le \ell<u\le 1$.  Then $\Phi=Q\circ H\circ T$ has a fixed point $f^\star\in Q(K_\alpha)$ by Theorem~\ref{thm:HITLFP}, and any bounded value–quantization of $Q$ yields a $\delta$–fixed point by Proposition~\ref{prop:AFPP}.  This provides a mathematically certified stabilized consensus in a standard HITL labelling/refinement loop on $L^1$.
\subsection*{Fairness and harm considerations}
The models studied in this section guarantee the existence (and approximate persistence) of a consensus artefact purely from geometric and analytic assumptions.  They are silent on questions of fairness, harm or bias.  In practice, collaborative human–AI workflows may amplify biases present in the training data or in human edits.  The operator $H$ may enforce ethical or legal constraints, and the coarsener $Q$ may reflect infrastructure that differentially affects groups.  While it is outside the scope of this note, designing operators $T$, $H$ and $Q$ that incorporate fairness objectives and bounding the propagated harm across iterations is an important direction for future work.

\medskip
\noindent\emph{Remark on uniqueness.}
If, in addition, the loop admits a uniform separation of alternatives (in the sense of the convexity–in–measure modulus underlying the normal–structure argument), then one can promote uniqueness of the stabilized artefact via the normal–structure machinery behind Lemma~4.9–Corollary~4.11 in the accompanying manuscript by Alpay and Alakkad\,\cite{AlpayAlakkadNew}.  An in‑depth exploration of these finer uniqueness issues lies outside the scope of the present note.

\subsection*{Illustrative numerical example}
To make the above theory more tangible, let us examine a simple discrete measure space $\Omega=\{1,2\}$ endowed with the counting measure $\mu(\{i\})=1$.  In this case $L^1(\mu)$ can be identified with $\mathbb R^2$ equipped with the $\ell^1$ norm.  For a fixed $\alpha\in(0,2]$ define
\[
K_{\alpha}=\bigl\{(a,b)\in[0,1]^2:\ a+b=\alpha\bigr\}.
\]
This set is a bounded, closed and convex subset of $\mathbb R^2$; because $\Omega$ is finite every bounded subset is compact in the topology of local convergence in measure.  Thus $K_{\alpha}$ satisfies the assumptions of our fixed point theorems.

Consider the following operators on $K_{\alpha}$.  Define a nonexpansive \emph{averaging} map
\[
T(a,b) := \Bigl(\tfrac{a+b}{2},\,\tfrac{a+b}{2}\Bigr),
\]
which maps each pair to its barycentre.  This map is easily seen to be nonexpansive in the $\ell^1$ norm.  For the human edit, fix $0\le \ell<u\le 1$ and set
\[
H(a,b) := (\min\{\max\{a,\ell\},u\},\, \min\{\max\{b,\ell\},u\}),
\]
which clips each component into the interval $[\ell,u]$.  This map is $1$–Lipschitz and idempotent.  Finally, let $q\colon[0,1]\to[0,1]$ denote rounding to the nearest multiple of $0.1$ and define the quantizer
\[
Q(a,b) := \bigl(q(a),q(b)\bigr).
\]
Rounding to a discrete grid is not in general nonexpansive—distances can increase by up to the rounding step—but the rounding error in each coordinate is bounded by $0.05$.  In particular, for any $(a,b)\in [0,1]^2$ one has $\|Q(a,b)-(a,b)\|_1\le 0.1$, and $Q$ satisfies $Q\circ Q = Q$.  Hence $Q$ fits the additive error model used in Proposition~\ref{prop:AFPP}.  The composite
\[
\Phi = Q\circ H\circ T
\]
therefore satisfies the hypotheses of Theorem~\ref{thm:HITLFP} (with $T$ nonexpansive and $Q$ idempotent).  A direct computation shows that $\bigl(q(\alpha/2), q(\alpha/2)\bigr)$ is a fixed point of $\Phi$; indeed, for any $(a,b)\in K_\alpha$ one has $T(a,b) = (\alpha/2, \alpha/2)$, $H$ leaves this point unchanged provided $\ell\le \alpha/2\le u$, and $Q$ rounds $\alpha/2$ to the nearest multiple of $0.1$.  For example, with $\alpha=1$ and $\ell=0$, $u=1$, starting from $(0.8,0.2)$ we have
\[
(0.8,0.2) \xrightarrow{T} (0.5,0.5) \xrightarrow{H} (0.5,0.5) \xrightarrow{Q} (0.5,0.5),
\]
so the iteration stabilises in a single step at the quantised consensus.  In general, any initial $(a,b)\in K_{\alpha}$ is mapped by $T$ to $(\alpha/2,\alpha/2)$ and then rounded, yielding a stable consensus artefact $(q(\alpha/2), q(\alpha/2))$.  This toy example illustrates on a finite measure space how our abstract results guarantee a stable consensus artefact and approximate fixed points in practice.


\section*{Conclusion}
We integrated the classical measure--compactness framework for fixed points in $L^1$ with an analysis of finite‑precision effects.  The measure‑compactness condition yields not only existence of genuine fixed points for nonexpansive maps, but also ensures that small perturbations (quantization) cannot eliminate fixed point phenomena entirely: one still finds approximate fixed points.  The pathological cases where $L^1$ fails the FPP correspondingly exhibit a failure of the AFPP under sufficiently fine quantization, reinforcing the necessity of the geometric assumptions (uniform integrability and normal structure) even when considering approximate orbits.  

Beyond these foundational results, we showed how human--AI interaction can be cast as a fixed point problem on $L^1$.  By modelling an AI proposal, a human edit and a coarse quantizer as successive nonexpansive operators on a measure‑compact set, we established the existence of a \emph{stable consensus artefact} for the resulting co‑editing loop and proved that this artefact survives as an approximate fixed point under bounded value–quantization error.  A concrete instantiation illustrates how these abstract principles yield mathematically certified stabilization in human‑in‑the‑loop labelling and refinement workflows.  Taken together, our work demonstrates that the geometric perspective of measure‑compactness and normal structure provides robust, interpretable guarantees for both theoretical fixed point problems and practical human–AI systems.
\paragraph{Limitations and future directions.}  The results presented here are existence statements: while our theorems guarantee that a stabilised artefact exists (and persists approximately under bounded quantization), they do not by themselves ensure uniqueness or address ethical questions such as fairness or potential harms in human–AI collaboration.  Uniqueness can sometimes be obtained under additional separation assumptions, as indicated in our remark, but further work is needed to characterise when a unique consensus must arise.  Our analysis also presupposes the measure–compactness of the underlying set $K$ and the nonexpansiveness of the component operators; when these assumptions fail, stabilisation may break down.  From a practical standpoint, it would be valuable to test the theory on larger empirical datasets and to explore how fairness constraints, bias mitigation and adversarial behaviour can be incorporated into the model.  We hope the illustrative example offered here encourages further numerical and empirical investigations alongside the geometric framework.

\section*{Appendix: Code for Numerical Examples}

While the results in this paper are purely theoretical, readers may find it helpful to experiment with the finite--precision examples.  The following Python code reproduces the quantisation counterexample from Section~2 and the illustrative human--AI loop from Section~\ref{sec:HITL}.  It can be run in any standard Python interpreter.

\smallskip
\noindent\textbf{Quantisation counterexample.}  This script implements the translation \(T(x)=x+d\,\mathrm{mod}\,1\) with \(d=0.6\) on the interval \([0,1]\), rounds the result to the nearest tenth (using the function \texttt{round(x*10)/10.0}) and iterates five times to reveal the period--5 cycle.

\begin{verbatim}
# Quantisation counterexample: translation plus rounding.
d = 0.6
delta = 0.05  # bound on rounding error

# Define the rounding operator to nearest tenth.
def Q(x):
    return round(x * 10) / 10.0

# Define the underlying translation operator.
def T(x):
    return (x + d) % 1.0

# Initialise the orbit at x0=0.0 and iterate five steps.
values = [0.0]
for n in range(5):
    x = values[-1]
    new_val = Q(T(x))
    values.append(new_val)
print("Orbit:", values)
\end{verbatim}

\noindent\textbf{Human--AI loop.}  This script simulates the two--point measure example with \(\alpha=1\).  The AI proposal averages the two components, the human edit clips values to the unit interval and the quantiser rounds to the nearest tenth.  Starting from \((0.8,0.2)\) the loop converges in one step to the consensus \((0.5,0.5)\), as illustrated in Section~\ref{sec:HITL}.

\begin{verbatim}
# Human--AI loop on a two-point measure space.
alpha = 1.0

# AI proposal: average the two components.
def T_pair(a, b):
    return ((a + b) / 2.0, (a + b) / 2.0)

# Human edit: clip each coordinate to [0,1].
def clip(u, lower=0.0, upper=1.0):
    return min(max(u, lower), upper)

def H_pair(a, b):
    return (clip(a), clip(b))

# Quantiser: round each coordinate to nearest tenth.
def Q_pair(a, b):
    return (round(a * 10) / 10.0, round(b * 10) / 10.0)

# Initialise (a,b) and iterate the composite Phi three times.
x = (0.8, 0.2)
for n in range(3):
    a, b = x
    x = Q_pair(*H_pair(*T_pair(a, b)))
    print(f"Iteration {n+1}:", x)
\end{verbatim}

\smallskip
These short scripts provide a transparent demonstration of the behaviours described in the text.  Researchers are encouraged to modify the parameters \(d\), \(\delta\), \(\alpha\) and the rounding granularity to explore how the dynamics change under different quantisation schemes.

%
\bigskip
\noindent\textbf{Higher--dimensional human--AI loop.}\
To illustrate how the human--AI consensus mechanism generalises, consider a four--point measure space (representing, for instance, four categories in a labelling task).  We identify $L^1$ with $\mathbb R^4$ equipped with the $\ell^1$ norm and normalise an arbitrary starting vector to have total mass $\alpha=2$.  The AI proposal averages the four components, the human edit clips values to $[0,1]$, and the quantiser rounds to the nearest tenth in each coordinate.  The following code selects a random initial vector, normalises it, and applies the composite map $\Phi=Q\circ H\circ T$ for several iterations.  One observes that the iterates converge quickly to a consensus vector in the quantised grid.

\begin{verbatim}
# Higher-dimensional human--AI loop on a four-point measure space.
import random

alpha = 2.0  # total mass

# AI proposal: average all components of a vector.
def T_vector(v):
    avg = sum(v) / len(v)
    return [avg for _ in v]

# Human edit: clip each coordinate to [0,1].
def H_vector(v, lower=0.0, upper=1.0):
    return [min(max(x, lower), upper) for x in v]

# Quantiser: round each coordinate to nearest tenth.
def Q_vector(v):
    return [round(x * 10) / 10.0 for x in v]

# Construct a random initial vector and normalise it to have sum alpha.
v = [random.random() for _ in range(4)]
total = sum(v)
v = [x * alpha / total for x in v]
print("Initial vector (normalised):", v)

# Iterate the composite operator Phi several times.
for n in range(5):
    v = Q_vector(H_vector(T_vector(v)))
    print(f"Iteration {n+1}:", v)
\end{verbatim}

\smallskip
This more complex example shows that the same averaging--clipping--quantising mechanism yields a consensus even in higher dimensions, consistent with the theoretical results of Section~\ref{sec:HITL}.  By adjusting the rounding grid or the number of categories, one can experiment with the effects of coarser or finer quantisation in multidimensional settings.

\subsubsection*{Conditional expectation example}
The preceding examples employ finite--range quantisers that round to a discrete grid.  To illustrate how a true conditional expectation fits into our framework, consider a finite four--point measure space $\Omega=\{1,2,3,4\}$ with the counting measure.  Identify $L^1(\Omega)$ with $\mathbb R^4$ equipped with the $\ell^1$ norm, and fix $\alpha\in(0,4]$.  Define
\[
  K_{\alpha}^{(4)}\;:=\;\Bigl\{ (a_1,a_2,a_3,a_4)\in[0,1]^4 : a_1+a_2+a_3+a_4=\alpha\Bigr\},
\]
which is bounded, closed, convex and measure--compact.  Partition $\Omega$ into two blocks $\mathcal P=\{\{1,2\},\{3,4\}\}$ and let $\mathcal G$ denote the corresponding finite $\sigma$--algebra.  The conditional expectation $Q=\mathbb E[\cdot\mid\mathcal G]$ acts on a vector $(a_1,a_2,a_3,a_4)$ by averaging coordinates within each block:
\[
  Q(a_1,a_2,a_3,a_4)\;=\;\Bigl( \tfrac{a_1+a_2}{2},\tfrac{a_1+a_2}{2},\tfrac{a_3+a_4}{2},\tfrac{a_3+a_4}{2}\Bigr).
\]
Define the AI proposal $T:K_{\alpha}^{(4)}\to K_{\alpha}^{(4)}$ by averaging all four coordinates,
\[
  T(a_1,a_2,a_3,a_4)\;=\;\Bigl( \tfrac{a_1+a_2+a_3+a_4}{4},\dots,\tfrac{a_1+a_2+a_3+a_4}{4}\Bigr),
\]
and let the human edit $H$ be simple clipping to the interval $[0,1]$ in each coordinate (which is nonexpansive).  The composite map $\Phi=Q\circ H\circ T$ is then a nonexpansive map on the measure--compact set $K_{\alpha}^{(4)}$.  By Theorem~\ref{thm:HITLFP} it has a fixed point $f^\star\in Q(K_{\alpha}^{(4)})$, and one can compute it explicitly: $f^\star$ assigns the average value $\alpha/4$ to all four points and then applies $Q$, yielding $f^\star=(\alpha/4,\alpha/4,\alpha/4,\alpha/4)$.  Iterating $\Phi$ from any starting point in $K_{\alpha}^{(4)}$ converges rapidly to this consensus vector.  This example demonstrates that our results apply not only to discretised quantisers but also to genuine conditional expectations with infinite range.

\end{document}